\documentclass[a4paper,reqno]{amsart}
\numberwithin{equation}{section}

\newtheorem{theorem}{Theorem}[section]

\newtheorem{lemma}[theorem]{Lemma}

\def\beq{\begin{equation}}
\def\eeq{\end{equation}}
\def\be{\begin{equation*}}
\def\ee{\end{equation*}}
\begin{document}
\title{The axiom of choice and model-theoretic structures}

\author{J.K. Truss}
  \address{Department of Pure Mathematics,
          University of Leeds,
          Leeds LS2 9JT, UK}
  \email{pmtjkt@leeds.ac.uk}



\begin{abstract}
We discuss the connections between the failure of the axiom of choice in set theory, and certain model-theoretic structures with enough symmetry.
\end{abstract}

\maketitle
%\tableofcontents

%%%%%%%%%%%%%%%%%%%%%%%%%%%%%%%%%%%%%%%%%%%%%%%%%%%%%%%%%%%%%%%%%%%%%%%%%%%%%%%%%%%%%
%%%%%%%%%%%%%%%%%%%%%%%%%%%%%%%%%%%%%%%%%%%%%%%%%%%%%%%%%%%%%%%%%%%%%%%%%%%%%%%%%%%%%

\section{\textbf{Introduction}}\label{intro}

This paper develops themes from \cite{Truss2}, where connections were described between questions about the axiom of choice and certain structures in model theory. Some of these ideas are also presented in 
\cite{Typke2} and \cite{Panasawatwong}. The main definition will be that two sets $X$ and $Y$ are {\em equivalent}, written $X \equiv Y$, if for any first order structure in a finite or countable language, 
which can be placed on the domain $X$, forming a structure $\mathcal X$, there is a first order structure (in the same language) which can be placed on $Y$ to form $\mathcal Y$, in such a way that 
$\mathcal X$ and $\mathcal Y$ are elementarily equivalent. Now it is a triviality that in the presence of the axiom of choice, two sets $X$ and $Y$ are equivalent in this sense if and only if they are 
finite of the same cardinality, or they are both infinite, as follows easily by using the L\"owenheim--Skolem Theorems. The situation however radically changes once one drops the axiom of choice. The theme 
will be that in the absence of the axiom of choice, even a mere {\em set} can hide quite a lot of (latent) structure. So, far from being more restrictive, dropping the axiom of choice can reveal features 
which in the presence of being able to well-order everything in sight become uninteresting.

To give an initial and fairly easy example from \cite{Truss3}, a set is said to be {\em amorphous} if it is infinite, but cannot be written as the disjoint union of two infinite sets. It was pointed out to 
the author by Azriel Levy that it is possible for an amorphous set $A$ to have no partition into infinitely many finite sets unless all but finitely many of them are singletons (this is called being {\em 
strictly amorphous}); it is also possible for an amorphous set $B$ to be the disjoint union of a family of 2-element sets. In the second case one can impose a graph structure on the set in which every 
connected component has size 2, but this is not possible for a strictly amorphous set. Hence these two are not `equivalent' in the sense just introduced. In other words, $B$ can be endowed with a structure
which makes it a graph with all connected components of size 2, but $A$ {\em cannot}, so that $A \not \equiv B$. Even though the graph structure is not `part of' $B$, {\em qua} set, it is somehow still 
potentially present, in a `latent' fashion.

In fact, the notion can be made more general, in the sense that one can consider other types of structure, not necessarily first order, on a set, and then one gets a similar definition, which may turn out 
to give rather different consequences. Whatever language $L$ one wishes to consider however, one would usually insist that the set of formulae or sentences of $L$ used should themselves admit a 
well-ordering, which gives a way of singling out one structure from others, as being the one which satisfies a sentence of $L$ which is least in this well-order that the other structures do not satisfy. We 
have a limited supply of examples of this type, sufficient may be to show that the idea could be fruitful. Specifically, in the case of Dedekind finite sets, being those which have no countably infinite 
subset, we can naturally subdivide into two cases, the first in which there is no {\em partition} of the set into a countably infinite family, and the rest. The former is called `weakly Dedekind-finite' in 
\cite{Typke1}, and the connections with $\aleph_0$-categorical structures are explained in that paper. So there is here a direct connection with (well studied) first order structures. For Dedekind finite 
sets outside this class however, one needs to consider stronger languages, such as first order languages with infinitely many symbols, or infinitary languages such as $L_{\omega_1, \omega}$ to gain any 
similar characterization.

\section{Fraenkel--Mostowski models}

The easiest way to form models of set theory without the axiom of choice is to use so-called {\em Fraenkel--Mostowski} models. These were invented by Fraenkel and Mostowski before Cohen showed how one can 
get models of set theory without the axiom of choice using forcing. The disadvantage is that they work in a weaker set theory, which we call FM, or possibly FMC (Fraenkel--Mostowski with choice) in which the 
axiom of extensionality is relaxed to allow the existence of `atoms', that is, sets which are not equal to the empty set, but all the same, have no members. The great advantage of using this method is that 
the arguments about groups of permutations are more transparent, and not complicated by technicalities involving automorphisms of the notions of forcing, names, and so on. The disadvantage is that one may 
wish for out and out Zermelo--Fraenkel consistencies. In many or most cases however, this is automatically provided by the Jech-Sochor Theorem \cite{Jech} or strengthenings due to Pincus \cite{Pincus1}, and 
this applies in all cases that we study here.  

The general framework for Fraenkel--Mostowski models used here is as follows. We start with a model $\mathfrak M$ for FMC in which there is a set $U$ of atoms (and saying that it satisfies FMC means that it 
fulfils all axioms of ZFC except for extensionality which now becomes the statement that two sets not in $U$ are equal if and only if they have the same members). We assume given a group $G$ of permutations 
of $U$ (lying in $\mathfrak M$), and a normal filter $\mathfrak F$ of subgroups of $G$. This means that $\mathfrak F$ is closed under finite intersections and supersets, and also under conjugacy by members 
of $G$. The group $G$ induces a natural action on the whole of $\mathfrak M$ by letting $g(a) = \{g(b): b \in a\}$, which is a valid definition by transfinite induction on rank, and the setwise and 
pointwise stabilizers of $a \in {\mathfrak M}$ are written as usual and defined by $G_a = \{g \in G: g(a) = a\}$ and $G_{(a)} = \{g \in G: (\forall b \in a)(g(b) = b)\}$ respectively. A set is said to be 
{\em symmetric} if its (setwise) stabilizer lies in $\mathfrak F$, and the resulting FM model is given by ${\mathfrak N} = \{x \in {\mathfrak M}: x \subseteq {\mathfrak N} \wedge G_x \in {\mathfrak F}\}$. 
Note that this is another definition by transfinite induction, which may be alternatively expressed by saying that all members of the transitive closure of $x$ are symmetric. Since we are dealing with FM 
models, in which extensionality is violated for the atoms, we note that for rank 0, $x$ can be either the `true' empty set, in which case $G_x = G$, so this automatically lies in $\mathfrak N$, or an atom 
$a$, which lies in $\mathfrak N$ provided that $G_a \in {\mathfrak F}$. It is usually assumed that the stabilizers of all members of $U$ lie in $\mathfrak F$, so this is also always true, and this ensures 
that $U \in {\mathfrak N}$. 

Note that the closure of $\mathfrak F$ under conjugacy is required to ensure that the axiom of replacement is true in $\mathfrak N$, which relies on the fact that $gG_xg^{-1} = G_{gx}$. The easiest way to 
ensure that this holds is to require that $\mathfrak F$ is `generated by finite supports', meaning that it is the filter generated by $\{G_a: a \in U\}$. Thus $H \in{\mathfrak F}$ if and only if 
$H \ge G_{(A)}$ for some finite $A \subseteq U$. Normality follows from $gG_{(A)}g^{-1} = G_{(gA)}$. More generally, it suffices for normality that $\mathfrak F$ is generated by the pointwise stabilizers of 
some family of sets which is fixed (setwise) by the action of $G$, as follows by the same calculation. (For instance, one may similarly consider the filter generated by countable supports, and so on.) In 
general, any $A \subseteq U$ satisfying $G_x \ge G_{(A)}$ is called a `support' of $x$.  A key point in discussing versions of AC in $\mathfrak N$ is that the pointwise stabilizer $G_{(x)}$ lies in 
$\mathfrak F$ if and only if $x$ can be well-ordered in $\mathfrak N$. In one direction, suppose that $G_{(x)} \in {\mathfrak F}$, and let $\prec$ be any well-ordering of $x$ in $\mathfrak M$. Then we can 
see that $G_{(x, \prec)} \ge G_{(x)}$, from which it follows that $x$ can be well-ordered {\em in $\mathfrak N$}. For if $g \in G_{(x)}$ and $y \prec z$, then $(y, z) \in \prec$, so from $gy = y$ and 
$gz = z$ we deduce that $(gy, gz) \in \prec$. Conversely, if $x$ can be well-ordered by $\prec$ in $\mathfrak N$, then $G_{(x, \prec)} \in {\mathfrak F}$. Since any automorphism of a well-ordered set is 
trivial, $G_{(x, \prec)}$ fixes each member of $x$, and hence $G_{(x, \prec)} \le G_{(x)}$, which implies that $G_{(x)} \in {\mathfrak F}$. To violate AC in $\mathfrak N$, the following property therefore 
suffices: no member of $\mathfrak F$ fixes the whole of $U$.

\section{Homogeneous structures}

Very often, the Fraenkel--Mostowski method is applied to a structure having a lot of symmetry, for instance $\mathcal A$ is countable and `homogeneous', meaning that any isomorphism between finitely 
generated substructures extends to an automorphism. (In the relational case, `finitely generated' is the same as `finite', which makes things easier to verify.) One takes the set $U$ of atoms to be indexed 
by the members of $\mathcal A$, the group to be induced on $U$ by the automorphism group of $\mathcal A$, and the filter generated by finite supports. The two most famous examples of this are Fraenkel's 
first model, in which the structure is an infinite set with just equality, and Mostowski's `ordered' model, in which $\mathcal A$ is the ordered set of rational numbers. In \cite{Truss2} (based on ideas in 
\cite{Pincus2}) this is applied to a series of classical implications/non-implications among weak versions of the axiom of choice, which were listed in Levy's paper \cite{Levy}. 
 
The connection between (countable) homogeneous structures and amalgamation classes was analyzed by Fra{\"\i}ss{\'e}. A family $\mathcal C$ of structures is said to be an {\em amalgamation class} if 
(1) it has only countably many members up to isomorphism, (2) it is closed under forming finitely generated substructures, and under isomorphisms, (3) it has the `joint embedding property' JEP: any two 
members can be embedded in another member, and (4) it has the {\em amalgamation property}, that if $B$, $C$, and $D$ are structures in $\mathcal C$, and embeddings $p_1$, $p_2$ of $B$ into $C$ and $D$ 
respectively are given, then there are $E$ in the class, and embeddings $p_3$, $p_4$ of $C$, $D$ respectively into $E$ such that `the diagram commutes', meaning that $p_3p_1 = p_4p_2$. Fra{\"\i}ss{\'e} 
showed that a class of finitely generated structures is an amalgamation class if and only if it is equal to the `age' of a countable homogeneous structure. Here the {\em age} of $\mathcal A$ is the family 
of structures which are isomorphic to a finitely generated substructure of $\mathcal A$. We state his main theorem as follows.

\begin{theorem} \label{3.1}:  Let $\mathcal C$ a class of finitely generated structures. Then 

(i) there is a countable structure $\mathcal A$ with age equal to $\mathcal C$ if and only if $\mathcal C$ satisfies properties (1), (2), and (3).

(ii) there is a countable homogeneous structure $\mathcal A$ with age equal to $\mathcal C$ if and only if $\mathcal C$ satisfies satisfies properties (1), (2), (3) and (4). Furthermore, this structure is 
unique up to isomorphism.    \end{theorem}

We omit the proof, but refer to \cite{Macpherson}. Instead, we give several examples, and explain the connection with the corresponding Fraenkel--Mostowski models in some cases. The unique homogeneous 
structure described in part (ii) of the theorem is called the {\em Fra{\"\i}ss{\'e} limit} of the corresponding amalgamation class.

(1) Any countably infinite set is the Fra{\"\i}ss{\'e} limit of the family of finite sets (with no structure apart from equality). This gives us the first Fraenkel model (which contains an `amorphous set'---see the next section).

(2) A countable dense linear order without endpoints is the Fra{\"\i}ss{\'e} limit of the class of finite linear orders. This gives the Mostowski `ordered' model. Mostowski showed that in this model, the axiom of choice is false, but any set can be linearly ordered (the `ordering principle'); later Halpern \cite{Halpern} showed that the boolean prime ideal theorem is also true in this model.

(3) The `random graph' is the Fra{\"\i}ss{\'e} limit of the class of all finite graphs.

(4) The countable atomless boolean algebra is the Fra{\"\i}ss{\'e} limit of the class of finite boolean algebras.

(5) The class of finite partial orders is an amalgamation class, and its Fra{\"\i}ss{\'e} limit is called the {\em generic partial order}. This was used in \cite{Mathias} to prove the independence of the 
order-extension theorem from the ordering principle; this used a Cohen model; to obtain the same result using a Fraenkel--Mostowki model one instead considers the Fra{\"\i}ss{\'e} limit of the class of finite
structures of the form $(X, <, \prec)$ where $<$ is a partial order, and $\prec$ is a linear extension of $<$.

(6) A vector space of countable dimension over a finite field $F$ is the Fra{\"\i}ss{\'e} limit of the class of finite vector spaces over $F$. This gives an amorphous set `of projective type' (see the next 
section).

(7) The Fra{\"\i}ss{\'e} limit of the class of finite bipartite graphs {\em with specified partition} is the generic bipartite graph. It may be best thought of as a partial order with two levels, i.e. in 
which there are no chains of length 3. 

Later we remark that the Fraenkel--Mostowski construction can be performed even for $\aleph_0$-categorical structures, noting that any homogeneous structure over a finite relational language is 
$\aleph_0$-categorical, but not necessarily conversely. Several of the models just mentioned (and others) are used in \cite{Panasawatwong} to analyze a variety of possible definitions of `finiteness'. 
Furthermore, Pincus \cite{Pincus1} showed how to use carefully constructed homogeneous models to give independences between versions of the axiom of choice, which streamlined earlier work solving individual 
cases.

\section{Amorphous sets and a notion of rank}

A set $X$ is said to be {\em amorphous} if it is infinite, but cannot be written as the disjoint union of two infinite sets. Superficially one might think that there isn't much to be said about amorphous 
sets, as the name suggests (`without form'). However, this is far from the case, and a study of the possibilities is carried out in \cite{Truss3}, with at least some success. It was pointed out to the 
author by Wilfrid Hodges that the definition of `amorphous' is very similar to that of `strongly minimal' in model theory, see \cite{Baldwin}. A (definable) set in a first order structure $\mathcal A$ is 
said to be {\em strongly minimal} if it is not the disjoint union of two infinite definable subsets (this is what `minimal' means), even if one is allowed to pass to elementary extensions of the given 
structure (this is what `strongly' means, so strictly speaking, it is a property of the formula defining the set, rather than of the set itself). This is clearly closely related to the definition of 
`amorphous', and since strongly minimal sets have been so widely studied, this observation gives rise to attractive families of amorphous sets. 

The three most famous examples of strongly minimal sets are a pure set, i.e. an infinite set with no structure on it (apart from the equality relation), a vector space over a finite field, and the field of 
complex numbers. Associated with any strongly minimal set is a geometry, and general strongly minimal sets are often described in terms of which kind of geometry they give rise to. The geometry for a pure 
set is uninteresting, and is called `disintegrated'. The one from a vector space is the associated projective geometry, and this is therefore called `of projective type'. The geometry arising from the 
complex numbers is a lot more complicated. 

There are significant differences between the theories of amorphous and strongly minimal sets. The first is that in set theory we are not constrained by the `first order' requirement, and so we should 
instead be thinking in terms of second or higher order definability; or more exactly, the version of definability which looks at properties invariant under the automorphism group. In view of the 
Ryll--Nardzewski Theorem \cite{Hodges}, this amounts to restricting to $\aleph_0$-categorical strongly minimal sets, so for instance, the most popular examples of strongly minimal sets such as the field of 
complex numbers, do not feature. However, examples arising from vector spaces over finite fields, which {\em are} $\aleph_0$-categorical, do arise (and are referred to as having `projective type'). 

The failure of the axiom of choice leads however to other families of examples of amorphous sets, which have no counterpart in the model-theoretic setting, those which are either `unbounded', or 
`unbounded of projective type'. So there is a strong interplay between the situations in set theory without choice and model theory, which does not however give a precise correspondence; strongly minimal 
sets are far richer by virtue of the non-$\aleph_0$-categorical cases, and amorphous sets are far richer by virtue of those of `unbounded' type. If $A$ is amorphous then we may consider what sorts of 
structures may be placed on it. The simplest one is an equivalence relation, which we analyze in terms of the possible partitions $\pi$ of $A$ it defines. If $\pi$ has only finitely many pieces, then it 
follows from $A$ amorphous that exactly one is infinite. Otherwise, if it has infinitely many pieces, then they must all be finite, and furthermore, all but finitely many of them must have the same size, 
which we refer to as the {\em gauge} of $\pi$. If $\pi$ has gauge $n$, then the union $X$ of the pieces of $\pi$ having other sizes is finite, and there is another partition of gauge $n$ such that $X$ has 
size $< n$. We can say that $\pi$ is in {\em standard form}. It is shown in \cite{Truss3} that any two partitions of $A$ of the same gauge have equal sizes of pieces `left over'. So, for instance, it is
not possible for an amorphous set $A$ to be expressible both as a disjoint union of pairs, and as a disjoint union of pairs together with a singleton (it must be either `even' or `odd'). In \cite{Truss3} 
we distinguish various cases. The simplest is that there is a greatest value of the possible gauges of partitions of $A$ into finite sets. This is called being bounded amorphous. 

The `classification' in \cite{Truss3} describes an amorphous set as being bounded or unbounded but not of projective type, or of projective type over (finite) fields of bounded or unbounded size. (Any 
amorphous set of projective type must be unbounded, but it is possible for the cardinalities of the underlying field to be bounded, or not.) This is only a very rough subdivision, rather than a true 
classification, but the aim is clear. If two such sets have the same classifiers, then we hope that in favourable instances the sets will be the `same'. Since one can only show that they are necessarily in 
bijective correspondence by use of the axiom of choice, the best one can hope for is that they are elementarily equivalent, once an appropriate language has been used to describe them. In other words, this 
is equivalence as introduced in the introductory section. 

To create positive examples, we may employ the Fraenkel--Mostowski method, introduced above. In fact we may do this rather generally. Let $\mathcal A$ be any structure, assumed to have a rich enough
automorphism group. In most cases, this will mean that it is `homogeneous', that is, any isomorphism between finite sets must extend to an automorphism (though $\aleph_0$-categoricity is actually 
sufficient). In particular, this will mean that its automorphism group $G$ acts singly transitively. In the Fraenkel-Mostowski construction, we therefore let the set of atoms be indexed by the elements of 
$\mathcal A$, and take as group of permutations the maps induced by automorphisms of $\mathcal A$, also written $G$. Finally, the filter of subgroups is generated by the pointwise stabilizers of finite 
sets.

In the amorphous case, let us consider some examples. In the first case, $\mathcal A$ is a countably infinite set containing an equivalence relation whose classes are all of size 2. The automorphism group 
is then the wreath product of the two-element group (swapping two elements of a pair) with the infinite symmetric group (which permutes the pairs). The structure we started with is strongly minimal, as any 
infinite set definable from finitely many parameters must contain an element such that it, and its paired element are not in the list of parameters, so it follows that the set must contain all elements
for which neither it nor its paired element is in the list of parameters (what is needed to do this formally is to note that the stabilizer of the set of parameters  acts transitively on the set of such 
points). 

In this model we could say that $|U|$ is an `even' number, meaning that $U$ is a disjoint union of pairs. As remarked above, it would then follow that $|U|$ cannot also be odd. However, the same model 
{\em does} contain `odd' amorphous sets, for instance, the one obtained by removing a point from $U$ (or adding one), or one can construct a model purpose built containing an odd amorphous set (though now 
with an $\aleph_0$-categorical but not homogeneous structure, being an equivalence relation in which all classes except one have size 2, the other being a singleton).

For an amorphous set of projective type, we can start with $\mathcal A$ being a vector space $V$ of dimension $\aleph_0$ over a finite field $F$, and $G$ its group of automorphisms (the general linear 
group). The fact that the resulting set is amorphous then follows by an easy group-theoretical argument as before. For if $X$ is a finite subset of $U$, the subspace generated by $X$ is finite, and the 
general linear group acts transitively on its complement. Variants of this example can produce amorphous sets actually corresponding to the geometry, that is the set of 1-dimensional subspaces of $V$; or one
can remove some points and work with what is left (with technicalities needed to sort this out). In other words, there are lots of different options obtained from essentially the same starting point. In all 
cases there will be an amorphous set of projective type, but it may be somewhat `hidden'. Note that in this case, there are partitions of $U$ into finite sets of arbitrarily large gauge, namely the cosets
of subspaces, but there is no way of `combining' these, which would provide a countable subset and violate amorphousness.

In \cite{Truss3} a method is presented of `recovering' the original structure from the structure in the model. A number of assumptions need to be made to make this work, but they are not unreasonable in the 
context. Specifically, one starts with a model $\mathfrak M$ of FMC containing a set $U$ of atoms, a group tailored to the particular case, which will here be the automorphism group of a structure, either an
equivalence relation with classes of size 2, or a vector space over a finite field, and the filter of finite supports, i.e. generated by the pointwise stabilizers of finite sets. One then forms the 
corresponding FM model $\mathfrak N$, which will contain an amorphous set of the appropriate type. Now imagine {\em starting} in $\mathfrak N$, and trying to return to $\mathfrak M$. We assume that $U$ is 
equipped with the appropriate structure, which in the cases we have so far described is either an equivalence relation with classes all of size 2, or a vector space over the given finite field. Once can 
recover the original choice model by adjoining a generic well-ordering of $U$ of type $\omega$. Conditions will be finite 1--1 sequences of members of $U$, with the natural partial order by extension, and 
then if $\mathfrak F$ is an $\mathfrak N$-generic filter for this partial order, by genericity it defines a well-ordering of $U$ of type $\omega$. Then one can see that the generic extension of 
$\mathfrak N$ by $\mathfrak F$ is actually equal to the original model $\mathfrak M$. 

An alternative method, which uses model theory rather than set theory, and in some ways is more flexible, was suggested by Cherlin, and is presented in \cite{Typke1}. Here one considers $U$ together with 
its structure in the two cases, and one looks at the first order theory true in $U$ in this structure. Although $U$ has no well-ordering (in $\mathfrak N$), the {\em language} used is well-orderable 
(actually the language is finite, but the set of formulae is well-orderable in type $\omega$). We need the following lemma. We write $\Delta$ for $\{x: \aleph_0 \not \le x\}$, the class of Dedekind finite 
cardinals, and $\Delta_4$ for $\{|X|: \mbox{ there is no map from $X$ onto } \omega\}$. The latter are called in \cite{Typke1} {\em weakly} Dedekind finite cardinals (even though the condition they have to 
fulfil is actually stronger than being Dedekind-finite).

\begin{lemma}\label{4.1} (i) For any cardinal number $x$, $x \in \Delta_4 \Leftrightarrow 2^x \in \Delta$.  

(ii) If $x, y \in \Delta$, then $x + y, xy \in \Delta$, and similarly for $x, y \in \Delta_4$.
\end{lemma} 

\begin{proof} (i) If $X$ has a partition into countably many non-empty sets, then certainly $\aleph_0 \le |{\mathcal P}(X)|$ so $2^x \not \in \Delta_4$. Conversely, suppose that $A_n$ are distinct subsets of 
$X$. For $Y \subseteq \omega$ let $A_Y = \bigcap_{n \in Y} A_n \cap \bigcap_{n \not \in Y}(X \setminus A_n)$. Then all the $A_Y$ are pairwise disjoint, and since $A_n = \bigcup\{A_Y: n \in Y\}$, and all 
$A_n$ are distinct, infinitely many $A_Y$ must be non-empty. Let ${\mathcal Y} = \{Y \in {\mathcal P}(\omega): A_Y \neq \emptyset\}$. Then $\mathcal Y$ is an infinite subset of ${\mathcal P}(\omega)$. Note 
that for any infinite subset $\mathcal Z$ of ${\mathcal P}(\omega)$, there is some $m$ such that $\{Z \in {\mathcal Z}: m \in Z\}$ and $\{Z \in {\mathcal Z}: m \not \in Z\}$ are both non-empty, and one of 
which must therefore be infinite. Using this, we may choose infinite subsets ${\mathcal Y}_n$ and $m_n \in \omega$ by induction. Let $m_0$ be the least such that $\{Y \in {\mathcal Y}: m_0 \in Y\}$ is an 
infinite proper subset of $\mathcal Y$ if any, or if not such that $\{Y \in {\mathcal Y}: m_0 \not \in Y\}$ is an infinite proper subset of $\mathcal Y$, and let ${\mathcal Y}_0$ be whichever of these two 
sets has been chosen (note that AC is {\em not} required here, since we made a choice which was definable in terms of the data). If $m_n$ and ${\mathcal Y}_n$ have been chosen, let $m_{n+1}$ be the least 
$m$ such that $\{Y \in {\mathcal Y}_n: m \in Y\}$ is an infinite proper subset of ${\mathcal Y}_n$ if any, or if not such that $\{Y \in {\mathcal Y}_n: m \not \in Y\}$ is an infinite proper subset of 
${\mathcal Y}_n$. Note that as all members of ${\mathcal Y}_n$ agree about membership or not of $m_0, \ldots, m_n$, we must have $m_n < m_{n+1}$. This defines an infinite strictly decreasing sequence of 
subsets of $X$, and hence there is a map from $X$ onto $\omega$.

(ii) First assume that $x, y \in \Delta$, and let $X$ and $Y$ be disjoint sets of cardinalities $x$ and $y$ respectively. If $f$ maps $\omega$ 1--1 into $X \cup Y$, then either $\{n: f(n) \in X\}$ or 
$\{n: f(n) \in Y\}$ is infinite, contrary to $x, y \in \Delta$. For products, suppose that $f$ maps $\omega$ 1--1 into $X \times Y$. Let $f(n) = (\xi_n, \eta_n)$. If $\xi_n$ takes infinitely many distinct 
values, this gives $\aleph_0 \le x$, and similarly, if $\eta_n$ takes infinitely many distinct values. Since $f$ takes infinitely many values, one of these must apply.  

Next suppose that $x, y \in \Delta_4$. If $g$ maps $X \cup Y$ onto $\omega$, then one of $g^{-1}(\omega) \cap X$ and $g^{-1}(\omega) \cap Y$ is infinite, contrary to $x, y \in \Delta_4$. For products, 
suppose that $f$ maps $X \times Y$ onto $\omega$. If for some $\xi \in X$, the image of $f$ on $\{\xi\} \times Y$ is infinite, we would have $y \not \in \Delta_4$. Hence each such image is a finite subset of
$\omega$. Since the finite subsets of $\omega$ can be well-ordered in type $\omega$, and $x \in \Delta_4$, only finitely many such sets can arise, which means that $f$ cannot be surjective after all.
\end{proof}
 
\begin{theorem}\label{4.2} \cite{Typke1} Let $U$ be a set with no countable partition, which admits a structure axiomatizable in a finite language $L$. Let $T$ be the (complete) set of sentences of $L$ which 
are true in $U$. Then for any $n$, the set of formulae with free variables among $\{v_0, v_1, \ldots, v_{n-1}\}$ is, up to $T$-equivalence, finite. Hence $T$ is $\aleph_0$-categorical.
\end{theorem} 

\begin{proof} For each formula $\varphi$ of $L$ in $\{v_0, v_1, \ldots, v_{n-1}\}$ look at all the $n$-tuples of members of $U$ satisfying $\varphi$. As $\varphi$ varies, this gives a countable family of 
subsets $U_\varphi$ of $U^n$. Now from the fact that there is no map from $U$ onto $\omega$, it follows that the same is true for $U^n$, which implies that its power set is Dedekind finite. Hence this 
family of sets is actually finite. If we say that $\varphi \sim \psi$ if $U_\varphi = U_\psi$, then 
$\varphi \sim \psi \Leftrightarrow U \models \forall v_0 \ldots \forall v_{n-1}(\varphi \leftrightarrow \psi) \leftrightarrow T \vdash \forall v_0 \ldots \forall v_{n-1}(\varphi \leftrightarrow \psi)$, and 
it follows that up to $T$-equivalence, there are only finitely many formulae in $\{v_0, v_1, \ldots, v_{n-1}\}$.

The statement we have established is well known to be an equivalent of $\aleph_0$-categoricity, and the equivalence does not require the axiom of choice in its proof.
\end{proof}

Applying this theorem, one can obtain results about the structures which certain Dedekind-finite sets can carry, appealing to known results with AC. An example given in \cite{Typke1} is that of what groups 
can exist on a set admitting a notion of `rank', as defined in \cite{Mendick}. We note that this notion has nothing to do with the usual set-theoretic notion of rank in the cumulative hierarchy, but is 
rather meant to be an analogue of Morley rank in a non-AC context. The definition is as follows. The empty set has {\em rank} $-1$. If $X \neq \emptyset$, then $X$ has {\em rank} $\alpha$, where $\alpha$ is 
an ordinal, if it does not have rank $< \alpha$, and there is some positive integer $k$ such that whenever $X$ is expressed as the disjoint union of $k+1$ sets, at least one of them has rank $< \alpha$. The
least such $k$ which will serve is called the {\em degree} of $X$. Note that this is precisely analogous to the usual definition of `Morley rank', where the definability requirements are omitted. One easily 
verifies that a set has rank 0 if it is non-empty and finite (and its degree is then equal to its cardinality). It is amorphous if and only if it has rank and degree 1. Thus `having rank' is a 
generalization of amorphous. One can indeed verify that any set having a rank in this sense is weakly Dedekind-finite, as is easily proved by transfinite induction. For let $\alpha$ be the least ordinal such
that some set $X$ of rank $\alpha$ (and degree $k$) is {\em not} weakly Dedekind-finite, and let $X = \bigcup_{n \in \omega}X_n$ be an expression for $X$ as a disjoint union of countably many non-empty sets.
Then we can write $X$ as the disjoint union of $k+1$ sets $Y_i = \bigcup_{n \in \omega}X_{n(k+1) + i}$, where $i \le k$, and at least one of these must have rank $< \alpha$, giving a set of smaller rank which
is not weakly Dedekind-finite, which is a contradiction. Now we can build FM models containing sets having arbitrarily large rank, and we can be curious as to what the structures are which are given to us 
by appealing to Theorem \ref{4.2}. We outline one way to do this, which seems to work rather more simply than the method given in \cite{Mendick}.

A boolean algebra is said to be {\em superatomic} if all its homomorphic images are atomic. An equivalent, and for us more useful, way of thinking of superatomic boolean algebras is given for example in 
\cite{Day}. Let $\mathbb B$ be any boolean algebra, and define an increasing sequence of ideals $I_\beta$ inductively thus. Let $I_0$ be the ideal generated by all atoms (where an {\em atom} is a minimal 
nonzero element), and assuming that $I_\beta$ has been defined, and is a proper ideal, let $I_{\beta + 1} \supseteq I_\beta$ be such that $I_{\beta + 1}/I_\beta$ is the ideal of $B/I_\beta$ generated by its 
atoms. At limits we take unions. Then $I_0 \subseteq I_1 \subseteq I_2 \subseteq \ldots$ is an increasing sequence of subsets of $\mathbb B$, so there is a least $\gamma$ such that 
$I_\gamma = I_{\gamma + 1}$. One can verify that $\gamma$ must be a successor, $= \beta + 1$ say, and $\beta$ is called the {\em rank} of $\mathbb B$. Then as is shown in \cite{Day}, $\mathbb B$ is 
superatomic if and only if ${\mathbb B}/I_\beta$ is finite. If we write its order as $2^k$, then $k$ is called the `degree' of $\mathbb B$. 

It is remarked in \cite{Hilton} that an efficient way to view superatomic boolean algebras, at any rate in the countable case, is via topologies on ordinals. If we consider ordinals of the form
$\omega^\alpha + 1$, where this is the ordinal power, and we consider successor ordinals for convenience, since then the usual ordinal topology is compact, we consider the boolean algebra $\mathbb B$ of 
clopen (closed and open) subsets. (To get a superatomic boolean algebra of rank $\alpha$ and degree $k$ one instead considers $\omega^\alpha \cdot k + 1$.) The operations are the usual ones of union, 
intersection, and complementation in the whole set. Then we can verify that $\mathbb B$ is a superatomic boolean algebra of rank $\alpha$ as follows. Notice that the atoms are precisely the singletons of 
the form $\{\beta\}$ where $\beta < \omega^\alpha$ is not a limit ordinal. In other words, $I_0$ consists of the finite sets of ordinals which are not divisible by $\omega$. In a similar way, $I_1$ comprises 
the finite sets of ordinals which are not divisible by $\omega^2$ and so on. In general, $I_\beta$ comprises all finite sets of ordinals which are not divisible by $\omega^{\beta + 1}$. Hence the least 
ordinal $\beta$ such that ${\mathbb B}/I_\beta$ is finite is $\alpha$. Another way of viewing the same thing is to note that the quotient algebra ${\mathbb B}/I_0$ corresponds to the set of limit ordinals, 
and the process of forming successive quotients precisely corresponds to taking the usual Cantor-Bendixson derivative. 
 
We now build an FM model based on such a superatomic boolean algebra $\mathbb B$. We let $U = \{u_\beta: \beta \le \omega^\alpha\}$, and we let $G$ be the group induced on $U$ by homeomorphisms of 
$X = \omega^\alpha + 1$ to itself. This time the filter used is the one generated by singleton supports as well as the stabilizers of clopen sets. Notice that we have to distinguish between atoms of the 
boolean algebra, and the members of $U$, which are `atoms' as in the FM construction. The latter include also $u_\lambda$ for limits $\lambda$, even though $\{\lambda\}$ is {\em not} clopen. We now need to 
demonstrate the connection with MT-rank. Here some care is needed, and this is because, for finite $\alpha$, $U$ {\em does} have rank $\alpha$, but not for infinite $\alpha$, in which case we need to pass 
to a suitable subset. In fact, for each $\lambda$ which is 0 or a limit ordinal, $Y_\lambda = \{u_\lambda, u_{\lambda + 1}, u_{\lambda + 2}, \ldots \}$ is (strictly) amorphous in $\mathfrak N$. First note 
that $Y_\lambda$ lies in $\mathfrak N$, since it is supported by the clopen set $\{u_{\lambda + 1}, u_{\lambda + 2}, \ldots, u_{\lambda + \omega}\}$ and the two singletons $\{u_\lambda\}$ and 
$\{u_{\lambda + \omega}\}$. For if $g \in G$ fixes each of these 3 sets, then it also fixes the boolean combination $Y_\lambda$. Now let $Y$ be a subset of $Y_\lambda$ in $\mathfrak N$, supported by 
$H \in {\mathfrak F}$, which contains the pointwise stabilizer of a finite subset of $U$ and finitely many clopen sets. Now any clopen set either intersects $Y_\lambda$ in a finite set or contains all but 
finitely many of its members, and so $H$ acts transitively on a cofinite subset of $Y_\lambda$. Hence $Y$ is finite or cofinite. Similarly we can see that for each $\lambda$ which is divisible by 
$\omega^2$, $Z_\lambda = \{u_{\lambda + \gamma}: \gamma < \omega^2\}$ lies in $\mathfrak N$ and has MT-rank 2, and so on. This establishes the statement about MT-rank in the case that $\alpha$ is finite. 

If however $\alpha \ge \omega$, then $U$ is not weakly Dedekind-finite, because for each $n \in \omega$, $\{u_\gamma: \gamma$ divisible by $\omega^n$ but not $\omega^{n+1}\}$ lies in $\mathfrak N$ with empty 
support, since $G$ preserves the topology, and these sets are definable from the topology via the Cantor--Bendixson derivative. So in this case, $U$ certainly does not have MT-rank. However, if we instead 
consider $X = \{u_\gamma: \gamma$ is not a limit ordinal$\}$, then we {\em do} obtain a set of MT-rank $\alpha$ (and degree 1) in all cases (even uncountable $\alpha$). To show this, we use transfinite 
induction. In fact we can show that for $0 < \beta < \alpha$ and $\gamma$ such that $\omega^\beta \cdot \gamma < \omega^\alpha$, $X \cap \{u_\delta: \delta < \omega^\beta\}$ and 
$X \cap \{u_\delta: \omega^\beta \cdot \gamma < \delta < \omega^\beta (\gamma + 1)\}$ are of MT-rank $\beta$ (and degree 1). First, they lie in $\mathfrak N$ since $\{u_\delta: \delta \le \omega^\beta\}$ and 
$\{u_\delta: \omega^\beta \cdot \gamma < \delta \le \omega^\beta(\gamma + 1)$ are clopen and the given subsets of $X$ differ from these by finite sets and by the set $\{u_\lambda: \lambda$ a limit$\}$, 
which is fixed by $G$. If $Y \subseteq X \cap \{u_\delta: \delta \le \omega^\beta\}$ lies in $\mathfrak N$, supported by a finite subset of $U$ and finitely many clopen sets, $Y$ contains or is disjoint 
from a final segment of $X \cap \{u_\delta: \delta \le \omega^\beta\}$, so by passing to its complement if necessary, suppose it is disjoint. Dividing into the cases $\beta$ a successor or limit, we see 
that $Y$ is contained in a finite union of sets corresponding to a smaller value of $\beta$, so has rank $< \beta$ by induction hypothesis, as required.

We conclude this section by giving some further examples of what structures can exists on amorphous sets, or indeed on ones having MT-rank. Even in \cite{Truss1} it was shown that no amorphous set can have 
an infinite subset which can be linearly ordered. Generalizing from this case, one can show that the same is true for any set having MT-rank. In fact, in \cite{Mendick} a precise characterization is given 
of which partial orders can exist on a set having MT-rank. With regard to other kinds of structure, such as groups, things are however rather different. We have seen that it is possible for an amorphous set 
of projective type to carry a non-trivial group structure, in fact even a vector space structure over a finite field. Groups arising in this way are however all abelian. Any set of size at least 6 can 
trivially carry a non-abelian group structure on a subset. One wonders whether this is all that can happen. The following theorem from \cite{Typke1} provides an elegant answer. A group is said to be 
{\em abelian-by-finite} if it has an abelian normal subgroup of finite index.

\begin{theorem}\label{4.3} A group $G$ whose underlying set has MT-rank is abelian-by-finite. 
\end{theorem} 

\begin{proof} The trick here is to deduce the result from \cite{Baur} where it is shown that an $\aleph_0$-categorical $\omega$-stable group is abelian-by-finite. By Theorem \ref{4.2}, the theory of
$G$ is $\aleph_0$-categorical, so has a unique countable model, written as $G^*$. So what remains is to show that $G^*$ is $\omega$-stable, and that `abelian-by-finite' can be expressed in a first order way,
so that it will be true also for $G$ (since $G$ and $G^*$ satisfy the same first order sentences). Now from the existence of MT-rank on $G$, it follows straightforwardly that $G^*$ has Morley rank (the 
precise argument is given in \cite{Typke1}), and hence that it is $\omega$-stable (see for instance \cite{Pillay} Proposition 6.46). It is also shown in \cite{Baur} that any $\aleph_0$-categorical 
abelian-by-finite group has an abelian subgroup of finite index which is definable without parameters, and hence the desired property is first order expressible.         \end{proof}

\section{Richer languages}

Up till now we have considered first order languages with finitely many symbols, and as we have seen this has restricted us very much to the weakly Dedekind finite case. We may relax this in two different 
ways. The first and less radical, is to allow the language to be infinite, that is, staying with a first order language, but allowing infinitely many symbols. If we are seeking to describe all members of 
$\Delta$, then we may also wish to consider second order, or infinitary logic, for instance $L_{\omega_1 \omega}$.

Let $L$ be a first order language with infinitely many symbols, but which can still be well-ordered. If there are infinitely many constant symbols, then these must be interpreted in any model, so its domain 
$A$ cannot possibly be Dedekind finite. Thus the first realistically different case for us is where there are infinitely many unary predicates ${\underline P}_n$. By taking boolean combinations, we may 
suppose that these are pairwise disjoint, and we therefore at once have infinitely many 1-types, so as expected, the theory will not be $\aleph_0$-categorical. For Russell's `$\omega$ pairs of socks', the 
set $P_n$ determined by each ${\underline P}_n$ has size 2, and these are required to be pairwise disjoint. If we use this to build a Fraenkel--Mostowski model using finite supports, then the automorphism 
group must fix each $P_n$, but it can arbitrarily interchange its two members. Using this, we can see that the set $U$ of atoms is Dedekind-finite in the resulting model, but $|U| \not \in \Delta_4$ since 
it keeps the countable partition. In fact, if $(X, <)$ is an infinite linearly ordered subset of $U$ in $\mathfrak N$, supported by $\bigcup_{n \le N}P_n$ for some $N$, then some member $x$ of $X$ lies 
outside this set, and the permutation $g$ which interchanges $x$ and the other member $y$ of the $P_m$ in which $x$ lies, and fixes all other points, lies in the stabilizer of $\bigcup_{n \le N}P_n$ but 
does not preserve the order. Thus $U$ has no infinite ordered subset, and hence no infinite well-ordered subset, so lies in $\Delta$.

Let us contrast this example with one described in section 4, which provided an amorphous set of gauge 2. The difference is that there the equivalence classes were not named. So if we instead say that 
$x \sim y$ if for some $m$, $x, y \in P_m$, then the automorphism group now just has to preserve the equivalence relation, so we can permute the sets $P_m$ for different $m$, and this leads to the wreath 
product as stated. Model-theoretically, the difference between the two structures is that for $(A, \{P_m: m \in \omega\})$, definable closure is locally finite, but algebraic closure is not. Here we 
recall that the {\em algebraic closure} $acl(X)$ of a set $X$ is the union of all the finite sets which are definable over $X$ (i. e. using parameters from $X$), whereas its {\em definable closure} $dcl(X)$ 
is the union of all the singleton sets which are definable over $X$. It is clear that $X \subseteq dcl(X) \subseteq acl(X)$. In $(A, \{{P_m}: {m \in \omega}\})$, the definable closure of any $X \subseteq A$ 
is the union of all the sets $P_m$ that $X$ intersects, whereas its algebraic closure is the whole of $A$, since each $P_m$ is a finite definable set. However, in $(A, \sim)$, the definable closure of $X$ is 
equal to its algebraic closure, and this is the union of all $P_m$ that $X$ intersects.

It is well-known, and easy to check, that in an $\aleph_0$-categorical structure, algebraic closure is necessarily locally finite, meaning that the algebraic closure of a finite set is finite. This means 
that we can use any $\aleph_0$-categorical structure $\mathcal A$ to produce a Fraenkel--Mostowski model in which the set of atoms is weakly amorphous, thereby giving the reverse construction as in Theorem 
\ref{4.2}, and generalizing the examples given at the end of section 3. The key point is that for any finite subset $X$ of $A$ (a support), its algebraic closure is finite, and by the Ryll--Nardzewski 
Theorem, the pointwise stabilizer of $X$ has only finitely many orbits on the complement of its algebraic closure. The examples previously given were homogeneous over a finite relational language, and these 
are necessarily $\aleph_0$-categorical, as one sees by counting the number of $n$-types for each $n$. 

To give an example of a non-homogeneous structure which is $\aleph_0$-categorical, consider a partially ordered set $(T, <)$ which is a {\em tree} (or semilinear order), meaning that it is not a linear 
order, any two elements have a common lower bound and for any $x \in T$, $\{y \in T: y \le x\}$ is linearly ordered. In \cite{Droste1}, a list of `sufficiently transitive' countable trees is given (where
this means that the automorphism group acts transitively on the set of 2-element chains, and also on the set of 2-element antichains). The simplest in Droste's list has maximal chains isomorphic to 
$\mathbb Q$, and every vertex splits $x$ into two `cones', meaning that $x$ is the greatest lower bound of two incomparable elements, but that for any three pairwise incomparable elements greater than $x$, 
there is a point above $x$ which lies below at least two of them. This characterization (together with countability) is sufficient to determine $T$ uniquely up to isomorphism, and its automorphism group is 
very rich. It is however not homogeneous, as it is easy to find two 4-element antichains (pairwise incomparable subsets) which are in distinct orbits of Aut($T$). Even in this case, we easily expand the 
language to make the structure homogeneous, as is done by adding in `join' as a new operation. Although this is no longer relational, it is still locally finite, which eases application of the Fra\"iss\'e 
theory. One can vary this example by considering instead the so-called `weakly 2-transitive trees' \cite{Droste2} (still in the countable case), being those in which the automorphism group acts transitively 
just on the family of 2-element chains. This time there are examples in which the tree has vertices for which the number of cones varies in an infinite set. This structure is definitely not 
$\aleph_0$-categorical, so if we use it to build a Fraenkel--Mostowski model, then its cardinality does not lie in $\Delta_4$, though one can easily check that it is still Dedekind-finite.

In all these cases we can attempt to find to what extent the (relational) structure $\mathcal A$ that we used when building the model is `detectable' from just the {\em set} $U$ which results, so that the 
set retains some `hidden' structure. The most straightforward instance is where the domain can be written (definably) as a disjoint union of $\aleph_0$-categorical sets, which can each be permuted 
independently of the others. So by expanding the language if necessary, we can write $A$ as the disjoint union of sets $P_n$ (thought of as unary predicates), each of which will be preserved by the 
automorphism group of $\mathcal A$ (the reason we insisted that $\mathcal A$ was relational was so that subsets would automatically give rise to induced substructures). Let $\mathfrak N$ be the resulting 
Fraenkel--Mostowski model. Then $U$ is the disjoint union of the interpretations $U_n$ of $P_n$. Furthermore, since $P_n$ is assumed $\aleph_0$-categorical, $U_n$ will be weakly Dedekind-finite. In the 
cases we consider, $U$ will be Dedekind-finite, which follows given transitivity of Aut($\mathcal A$) on each $P_n$. In fact we can verify that definable closure is locally finite. For let $X \subseteq A$ 
be finite. Then for some $n$, $X \subseteq \bigcup_{i < n}P_i$. Since the restriction of $\mathcal A$ to $P_i$ is $\aleph_0$-categorical, the algebraic closure of $X \cap P_i$ is finite, and hence so is its 
definable closure, and since Aut($\mathcal A$) acts transitively on each $P_n$, the overall definable closure of $X$ is finite. This means that given any finite support for an infinite subset of $U$, there 
must be an infinite orbit (non-trivial would suffice) of its stabilizer, so this subset cannot be well-orderable in the model. Without transitivity this would fail, since for instance, each $P_i$ could be 
an infinite dense linear ordering with endpoints, which would be $\aleph_0$-categorical but not transitive, and the set of left endpoints would form a countable subset. Also, if the pieces cannot be 
permuted independently, the argument would also fail. For instance, if every $P_i$ is a copy of $\mathbb Q$, and the group is order-preserving on each, but is required to map 0 in each copy in the same way,
then the whole of the group acts transitively on each $P_i$, but the stabilizer of 0 in one copy fixes 0 in all copies, yielding a countable subset.

\section{Remarks on equivalence}

In this final section we return to the equivalence relation which was mentioned at the start, and use some of the examples given to elucidate it. We recall that  we say that sets $X$ and $Y$ are 
{\em equivalent}, if for any first order language $L$ having countably many symbols, for any $L$-structure that can be put onto $X$, there is an elementarily equivalent $L$-structure that can be put on $Y$. 
As remarked before, in the presence of the axiom of choice, all infinite sets are equivalent. Even without choice, it will still be the case that all infinite well-orderable sets are equivalent (by the same 
proof, using L\"owenheim--Skolem; one needs to check that the existence of a model of arbitrarily large cardinality may be proved using a standard technique, for instance a Henkin style proof, using an 
effective well-ordering of the formulae involved in terms of the originally given well-ordering of the set, and the language). Furthermore, it is obvious that If $X_1 \equiv Y_1$ and $X_2 \equiv Y_2$, and 
$X_1 \cap X_2 = Y_1 \cap Y_2 = \emptyset$, then $X_1 \cup X_2 \equiv Y_1 \cup Y_2$, which means that we can immediately find enormous $\equiv$-classes (proper classes) of even non-well-orderable sets. One 
might wonder whether there are any `small' equivalence classes. Since any two sets having the same cardinality are necessarily equivalent, any $\equiv$-class not containing the empty set is automatically a 
proper class, so instead we should say that an equivalence class is `small' if the set of {\em cardinalities} of its members is small.

\begin{theorem}\label{6.1} (i) If $X$ is amorphous, and $X \equiv Y$, then $Y$ is also amorphous.

(ii) $X$ is Dedekind-finite, and $X \equiv Y$, then $Y$ is also Dedekind-finite.
\end{theorem} 

\begin{proof} (i) Suppose for a contradiction that $Y$ is not amorphous, and write $Y$ as a disjoint union $Y = Y_1 \cup Y_2$. Let $L$ the language $\{=, P\}$ where $P$ is a unary predicate, interpreted in
$Y$ as $Y_1$. Since $X \equiv Y$, there is an interpretation $X_1$ in $X$ for $P$. The facts that $Y_1$ and $Y \setminus Y_1$ are infinite are each expressible by an (infinite) set of first order sentences,
so it follows that $X_1$ and $X \setminus X_1$ must also be infinite, contrary to amorphousness of $X$.

(ii) This time we can use the language $L$ containing $=$ and one unary function symbol $f$. If $X \equiv Y$ where $Y$ is not Dedekind-finite, then there is a 1--1 function from $Y$ to $Y$ which is not onto,
which is expressible as an $L$-sentence. In any interpretation of $L$ in $X$, we would have a 1--1 but not onto function, contrary to $X$ Dedekind-finite.
\end{proof}
 
We remark that similar results hold for most of the notions of finiteness discussed in \cite{Panasawatwong}.

Even here, it doesn't follow that the $\equiv$-class of an amorphous set is small. In fact, it is consistent that there is a proper class of equivalent amorphous sets of distinct cardinalities (for 
instance, they could all be strictly amorphous), but it is also consistent that there is an amorphous set with small equivalence class. For instance, in the Fraenkel model for a single (strictly) amorphous 
set, one can easily check by a support argument that any amorphous set has cardinality which differs from the original one by a finite number (plus or minus), so actually in this case the equivalence class 
is countable. 

As remarked above, all infinite well-orderable sets are equivalent, and from this we can deduce the existence of many `large' equivalence classes, based on the above observation concerning unions. For 
instance, if $X$ is an amorphous set, then the unions of $X$ with all infinite well-ordered sets are all equivalent, again giving a proper class of possibilities. This idea is exploited in the following 
result.

\begin{theorem}\label{6.2} Any equivalence class of infinite sets which is a set must consist entirely of Dedekind finite sets.  \end{theorem} 

\begin{proof} Let $X$ lie in an equivalence class which is a set, and suppose for a contradiction that $|X| \not \in \Delta$. Then we can write as a disjoint union $X = Y \cup Z$ where $|Z| = \aleph_0$. 
By the L\"owenheim--Skolem Theorem, $Z$ is equivalent to all infinite ordinals, and hence on taking unions, $X$ is equivalent to the union of $Y$ with any infinite ordinal so there is a proper class of 
cardinals of sets equivalent to $X$, contradicting our assumption.   \end{proof}

This theorem explains the main emphasis in this paper on Dedekind finite sets. It doesn't mean however that the non-Dedekind finite case need necessarily be uninteresting, but it is so far unexplored. 

The way in which the definition of equivalence works out is clearly heavily dependent upon which language or type of language we allow. We conclude with a number of examples which illustrate some of the 
possibilities.

\vspace{.1in}

\noindent{\bf Example 6.3}:
Consider the `pairs of socks' example from section 5. Here $U = \bigcup_{n \in \omega}P_n$. Then $U$ is equivalent to any union of infinitely many $P_n$, as any such can be made into a model of the same set 
of sentences on reinterpreting the predicates ${\underline P}_n$. Hence the equivalence class of $U$ has size at least $2^{\aleph_0}$. Now consider any $X \equiv U$. Then putting the structure given by the 
${\underline P}_n$ on $X$, we deduce that there must exist pairwise disjoint subsets $Q_n$ of $X$ also all of size 2. It is not clear however that $X = \bigcup_{n \in \omega}Q_n$. In ordinary first order 
model theory, we would definitely {\em not} know this, as by the compactness theorem this theory has nonstandard models, whose elements which do not lie in $\bigcup_{n \in \omega}Q_n$ all have the same 
1-type. We can try expanding the language to rule this out here by adding one more unary predicate $\underline Q$ to the language to stand for $X \setminus \bigcup_{n \in \omega}Q_n$. Since $U \equiv X$,
there must be an interpretation for all the ${\underline P}_n$ and $\underline Q$ in $U$. If we could guarantee that ${\underline P}_n$ is still interpreted by $P_n$, then this would at once yield a 
contradiction, but of course we do not know this is the case. So it is slightly mysterious which $X$ can arise here.

\vspace{.1in}

\noindent{\bf Example 6.4}
The same Fraenkel-Mostowski model results if we let $G$ be the group of permutations of $U$ which preserves the partial ordering given by $u \prec v$ if for some $m < n$, $u \in P_m$ and $v \in P_n$. The 
difference here is that we are now working in a language with finitely many symbols. For the same reason as in the previous example, there are at least $2^{\aleph_0}$ (cardinals of) sets equivalent to $U$.
Running through the same discussion, let $X \equiv U$. Thus there must exist a partial order on $X$ satisfying precisely the same sentences as $(U, \prec)$ did. In particular, every element lies in an 
antichain of size 2. Also there are elements on the bottom level, and the next, and the next ... (as these can all be expressed by first order sentences), so $X$ begins with a copy of $(U, \prec)$, but 
apparently may have other elements above. Once again, if we attempt to rule them out by introducing a new unary predicate $\underline Q$ for the `infinite' elements, then it may be possible to reinterpret
$\prec$ in $U$ to mirror this situation, by ordering the levels of which it comprised in a different way.

\vspace{.1in}

\noindent{\bf Example 6.5} We can vary the basic strictly amorphous construction at higher cardinalities $\kappa$. Let $\lambda$ and $\kappa$ be infinite (well-ordered) cardinals, with 
$\lambda \le \kappa$. We let $U = \{u_\alpha: \alpha < \kappa\}$, $G$ be the group of all permutations of $U$, and take the filter $\mathfrak F$ of subgroups of $G$ generated by the pointwise stabilizers of 
subsets of $U$ of cardinality $< \lambda$. Then in the resulting model, $U$ cannot be written as the disjoint union of two non-well-orderable sets. In fact we have included $\lambda$ in the description only
to illustrate a particular point, namely that there is no essential difference between the models formed by taking $\lambda$ and $\kappa$, and $\lambda$ and $\lambda$. As far as the model is concerned, a
subset of $U$ is either of cardinality $< \lambda$, or its complement is. For instance, if $\lambda = \omega$ we again obtain a strictly amorphous set, which is equivalent to the one found by also taking
$\kappa = \omega$. So, now assuming that $\lambda = \kappa > \omega$, we see that in order to `describe' $U$, we definitely appear to require a stronger language than just first order logic with countably 
many symbols. There are (at least) two main options; one is to add a large number of symbols, probably constants in this case; the other is to use a second order language. Adopting the second idea in the 
special case where $\kappa = \omega_1$, we can express the main features by saying that there is an infinite well-orderable subset, and any two infinite well-orderable subsets have the same cardinality. We
can also say that $X$ cannot be well-ordered, but whenever it is written as a disjoint union $Y \cup Z$, one of $Y$ and $Z$ is countable. These properties can all be expressed in second order logic. For 
larger values of $\kappa$, things are more complicated, and what can be expressed, even in second order logic, will depend to what extent $\kappa$ is definable.

\vspace{.1in}

\noindent{\bf Example 6.6} 
We can form a transfinite version of pairs of socks, in which $U = \bigcup_{\alpha < \kappa}P_\alpha$, where the $P_\alpha$ are pairwise disjoint 2-element sets. Here the group has to fix every $P_\alpha$,
and the supports are taken to be subsets of $U$ of cardinality $< \kappa$. In this model, $U$ has subsets of all cardinalities $< \kappa$, but not of cardinality $\kappa$ itself, though there is a surjection 
to $\kappa$. The situation can be described in second order logic, depending on the value of $\kappa$, or alternatively using first order logic with $\kappa$ unary predicates ${\underline P}_\alpha$. 
Similar remarks apply about its equivalence class as in Example 6.3. 

\vspace{.1in}

\noindent{\bf Example 6.7}  
Finally we may vary the vector space example, and take $U$ to be indexed by a $\kappa$-dimensional vector space $V$ over a finite field, or the rationals. The group $G$ is induced by the automorphism group
of $V$ (its general linear group), and we may take $\mathfrak F$ to be the filter of subgroups generated by the pointwise stabilizers of countable sets. This has properties analogous to example 6.5 
vis-a-vis the amorphous case, in that any subset of $U$ in the model is either contained in a countable subspace, or contains the complement of such a subspace. Again a lot can be expressed about this 
situation using second order logic, but exactly what will depend on the value of $\kappa$. The natural choice of language will have, in addition to vector addition, zero vector, and additive inverses, and a
unary function symbol signifying multiplication by each member of the field.

%%%%%%%%%%%%%%%%%%%%%%%%%%%%%%%%%%%%
%% Bibliography
%%%%%%%%%%%%%%%%%%%%%%%%%%%%%

\end{document}